\documentclass[a4paper,12pt]{amsart}
\usepackage{amssymb}

\usepackage[colorlinks=true, linktocpage=false]{hyperref}
\numberwithin{equation}{section}
\newtheorem{theorem}{Theorem}[section]
\newtheorem{prop}[theorem]{Proposition}
\newtheorem{lem}[theorem]{Lemma}

\theoremstyle{remark}

\newcommand{\R}{\mathbb{R}}

\newcommand{\N}{\mathbb{N}}

\author[M.~Ltifi]{Jamel benameur And Maroua Ltifi}
\address{Department of Mathematics, Faculty of Science of Gab\`es, university of Gab\`es; Tunisia}
\email{\sl Jamelbenameur@gmail.com}
\email{\sl widaltifi@gmail.com}

\title[Strong solution of modified anistropic 3D-Navier-Stokes equations ]
{Strong solution of modified anistropic 3D-Navier-Stokes equations}

\begin{document}
	\begin{abstract}
		In this paper we study the anisotropic incompressible Navier-Stokes equations with a logarithm damping $\alpha \log(e+|u|^2)|u|^2u$ in $H^{0.1}$, where we used new methods, new tools and Fourier analysis.
	\end{abstract}

	

	\maketitle
	\tableofcontents

	
	\section{\bf Introduction}
	The goal of this paper is to investigate the modified anisotropic Navier-Stokes system shown below
	$$(NSAn_{\log})_{\alpha}
	\begin{cases}
		\partial_t u
		-\nu\Delta_h u+ u.\nabla u  +\alpha \log(e+|u|^2)|u|^2u =\;\;-\nabla p\hbox{ in } \mathbb R^+\times \mathbb R^3\\
		{\rm div}\, u = 0 \hbox{ in } \mathbb R^+\times \mathbb R^3\\
		u(0,x) =u^0(x) \;\;\hbox{ in }\mathbb R^3.
	\end{cases}
	$$
		where $u=u(t,x)=(u_1,u_2,u_3)$ and $p=p(t,x)$ denote respectively the unknown velocity and the unknown pressure of the fluid at the point $(t,x)\in \mathbb R^+\times \mathbb R^3$, $\alpha>0$, $\Delta_{h}=\partial^{2}_1+\partial^{2}_2$ and $\partial_{i}$ denotes
		the partial derivative in the direction $x_i$. The terms $(u.\nabla u):=u_1\partial_1 u+u_2\partial_2 u+u_3\partial_3u$, while $u^0=(u_1^o(x),u_2^o(x),u_3^o(x))$ is an initial given velocity. If $u^0$ is quite regular, the divergence free condition determines the pressure $p$. We recall in our case it was assumed the viscosity is unitary ($\nu=1$) in order to simplify the calculations and the proofs of our results.\\
	Clearly, when $\alpha=0$ it is corresponds to the classical anisotropic Navier-Stockes equation for more details the reader is referenced to the book \cite{JP} and \cite{R}.\\
the first step of this work is to study the modified Navier-stockes for $\beta>3$ :
$$	\begin{cases}
		\partial_t u
		-\Delta_{h} u+ u.\nabla u  +\alpha |u|^{\beta-1}u =\;\;-\nabla p\hbox{ in } \mathbb R^+\times \mathbb R^3\\
		{\rm div}\, u = 0 \hbox{ in } \mathbb R^+\times \mathbb R^3
	\end{cases}
	$$
	which is exemplified by the following theorem.
	\begin{theorem}\label{theo1}
			Let $u^0\in H^{0.1}(\mathbb R^3)$ be a divergence free vector fields. For $\beta>3$  there is a unique global solution $u\in L^\infty(\R^+,H^{0.1}(\mathbb R^3)\cap C(\R^+,H^{-2}(\R^3))\cap L^2(\R^+,\dot{H}^1(\mathbb R^3))$ such that
	\begin{equation}\label{eqth01}\|u(t)\|_{L^2}^2+2\int_0^t\|\nabla_{h} u\|_{L^2}^2+2\alpha\int_0^t\|u\|_{L^{\beta+1}}^{\beta+1}\leq \|u^0\|_{L^2}^2.\end{equation}
\begin{align}\label{eqth02}\nonumber
	\|\partial_{3} u(t)\|_{L^2}^2+2\int_0^t\|\nabla_{h}\partial_{3} u\|_{L^2}^2+\alpha(\beta-1) \int_{0}^{t}\||u|^{\beta-3} |\partial_{3}|u_n|^2|^2\|_{L^{1}}\\+2\alpha\int_0^t\||u|^{\beta-1} \partial_{3} |u|^2\|_{L^1}\leq\|\partial_{3}u^0\|_{L^2}+16\int_0^t\||u|^{2} \partial_{3} |u|^2\|_{L^1}.\end{align}
	\end{theorem}
Despite this, the continuity and uniqueness of this anisotropic modified equation remain a major unresolved issue for $\beta=3$ Indeed, because of the inequality $(\ref{eqth02})$, the problem is limited to the case $0<\alpha<8$.
The function $\log(e+|u|^2)$ will be included in our statement as a trick. So we'll look at the limiting case $\beta=3$ as second step of this paper.\\The following theorem illustrates the main result of our work:
	\begin{theorem}\label{theo2} Let $u^0\in H^{0,1}(\mathbb R^3)$ be a divergence free vector fields, then there is a unique global solution of $(NSAn_{\log})_{\alpha}$:
	$u\in L_{loc}^\infty(\R^+,H^{0,1}(\mathbb R^3)\cap C(\R^+,L^{2}(\R^3))\cap L_{loc}^2(\R^+,\dot H^1(\mathbb R^3))$ and $\log(e+|u|^2)|u|^4,\frac{|u|^2}{e+|u|^2} |\partial_{3}|u|^2|^2,  \log(e+|u|^2) |\partial_{3}|u|^2|^2,\log(e+|u|^2)|u|^2 |\partial_{3} u|^2\in L^1_{loc}(\R^+,L^1(\R^3))$. Moreover, for all $t\geq0$ \begin{equation}\label{eqth1}\|u(t)\|_{L^2}^2+2\int_0^t\|\nabla_{h} u\|_{L^2}^2+2\alpha\int_0^t\|\log(e+|u|^2)|u|^4\|_{L^1}\leq \|u^0\|_{L^2}^2.\end{equation}
	\begin{align}\label{eqth2}
		\nonumber\| \partial_{3} u(t)\|_{L^2}^2+2\int_0^t\|\nabla_{h}\partial_{3} u\|_{L^2}^2+\alpha\int_0^t\|\frac{|u|^2}{e+|u|^2} |\partial_{3}|u|^2|^2|\|_{L^1}\\+\alpha  \int_0^t\|\log(e+|u|^2) |\partial_{3}|u|^2|^2\|_{L^1}+\alpha  \int_0^t\|\log(e+|u|^2)|u|^2 |\partial_{3} u|^2\|_{L^1}&\leq \|\partial_{3} u^0\|_{L^2}^2e^{b_{\alpha}t} ,\end{align}
	where $b_{\alpha}=e^{\frac{3}{\alpha}}-e$
\end{theorem}
The rest of our paper is structured as follows. We present some notations, definitions, and preliminary results in the second section. We will look at the global solution of Theorem \ref{theo2} in Section 3. Furthermore, the solution's uniqueness and right continuity.
	\section{\bf Notations and preliminary results}
\subsection{Notations} In this section, we collect some notations and definitions that will be used later.\\
\begin{enumerate}
	\item[$\bullet$] The Fourier transformation is normalized as
	$$
	\mathcal{F}(f)(\xi)=\widehat{f}(\xi)=\int_{\mathbb R^3}\exp(-ix.\xi)f(x)dx,\,\,\,\xi=(\xi_1,\xi_2,\xi_3)\in\mathbb R^3.
	$$
	\item[$\bullet$] The inverse Fourier formula is
	$$
	\mathcal{F}^{-1}(g)(x)=(2\pi)^{-3}\int_{\mathbb R^3}\exp(i\xi.x)g(\xi)d\xi,\,\,\,x=(x_1,x_2,x_3)\in\mathbb R^3.
	$$
	\item[$\bullet$] The convolution product of a suitable pair of function $f$ and $g$ on $\mathbb R^3$ is given by
	$$
	(f\ast g)(x):=\int_{\mathbb R^3}f(y)g(x-y)dy.
	$$
	\item[$\bullet$] If $f=(f_1,f_2,f_3)$ and $g=(g_1,g_2,g_3)$ are two vector fields, we set
	$$
	f\otimes g:=(g_1f,g_2f,g_3f),
	$$
	and
	$$
	{\rm div}\,(f\otimes g):=({\rm div}\,(g_1f),{\rm div}\,(g_2f),{\rm div}\,(g_3f)).
	$$
	Moreover, if $\rm{div}\,g=0$ we obtain
	$$
	{\rm div}\,(f\otimes g):=g_1\partial_1f+g_2\partial_2f+g_3\partial_3f:=g.\nabla f.
	$$
	\item[$\bullet$] Let $(B,||.||)$, be a Banach space, $1\leq p \leq\infty$ and  $T>0$. We define $L^p_T(B)$ the space of all
	measurable functions $[0,t]\ni t\mapsto f(t) \in B$ such that $t\mapsto||f(t)||\in L^p([0,T])$.\\
	\item[$\bullet$] The Sobolev space $H^s(\mathbb R^3)=\{f\in \mathcal S'(\mathbb R^3);\;(1+|\xi|^2)^{s/2}\widehat{f}\in L^2(\mathbb R^3)\}$.\\
	\item[$\bullet$] The homogeneous Sobolev space $\dot H^s(\mathbb R^3)=\{f\in \mathcal S'(\mathbb R^3);\;\widehat{f}\in L^1_{loc}\;{\rm and}\;|\xi|^s\widehat{f}\in L^2(\mathbb R^3)\}$.\\
	\item[$\bullet$] For $R>0$, the Friedritch operator $J_R$ is defined by
	$$J_R(D)f=\mathcal F^{-1}({\bf 1}_{|\xi|<R}\widehat{f}).$$
	\item[$\bullet$] The Leray projector $\mathbb P:(L^2(\R^3))^3\rightarrow (L^2(\R^3))^3$ is defined by
	$$\mathcal F(\mathbb P f)=\widehat{f}(\xi)-(\widehat{f}(\xi).\frac{\xi}{|\xi|})\frac{\xi}{|\xi|}=M(\xi)\widehat{f}(\xi);\;M(\xi)=(\delta_{k,l}-\frac{\xi_k\xi_l}{|\xi|^2})_{1\leq k,l\leq 3}.$$
	\item[$\bullet$] $L^2_\sigma(\R^3)=\{f\in (L^2(\R^3))^3;\;{\rm div}\,f=0\}$.
	\item[$\bullet$] $\dot H^1_\sigma(\R^3)=\{f\in (\dot H^1(\R^3))^3;\;{\rm div}\,f=0\}$.
	\item[$\bullet$] $C_{r}(I,B)=\{f:I\rightarrow B\mbox{ right continuous }\}$ , where $B$ is Banach space and $I$ is an interval.
	\item[$\bullet$] Let $a\in\R,$ we define $a_{+}=\max(a,0)$.
\end{enumerate}
\subsection{Preliminary results}
In this section, we recall some classical results and we give new technical lemmas.
\begin{prop}(\cite{HBAF})\label{prop1} Let $H$ be Hilbert space.
	\begin{enumerate}
		\item If $(x_n)$ is a bounded sequence of elements in $H$, then there is a subsequence $(x_{\varphi(n)})$ such that
		$$(x_{\varphi(n)}|y)\rightarrow (x|y),\;\forall y\in H.$$
		\item If $x\in H$ and $(x_n)$ is a bounded sequence of elements in $H$ such that
		$$(x_n|y)\rightarrow (x|y),\;\forall y\in H.$$
		Then $\|x\|\leq\liminf_{n\rightarrow\infty}\|x_n\|.$
		\item If $x\in H$ and $(x_n)$ is a bounded sequence of elements in $H$ such that
		$$\begin{array}{l}
			(x_n|y)\rightarrow (x|y),\;\forall y\in H\\
			\limsup_{n\rightarrow\infty}\|x_n\|\leq \|x\|,\end{array}$$
		then $\lim_{n\rightarrow\infty}\|x_n-x\|=0.$
	\end{enumerate}
\end{prop}
\begin{lem}(\cite{JYC})\label{LP}
	Let $s_1,\ s_2$ be two real numbers and $d\in\N$.
	\begin{enumerate}
		\item If $s_1<d/2$\; and\; $s_1+s_2>0$, there exists a constant  $C_1=C_1(d,s_1,s_2)$, such that: if $f,g\in \dot{H}^{s_1}(\mathbb{R}^d)\cap \dot{H}^{s_2}(\mathbb{R}^d)$, then $f.g \in \dot{H}^{s_1+s_2-1}(\mathbb{R}^d)$ and
		$$\|fg\|_{\dot{H}^{s_1+s_2-\frac{d}{2}}}\leq C_1 (\|f\|_{\dot{H}^{s_1}}\|g\|_{\dot{H}^{s_2}}+\|f\|_{\dot{H}^{s_2}}\|g\|_{\dot{H}^{s_1}}).$$
		\item If $s_1,s_2<d/2$\; and\; $s_1+s_2>0$ there exists a constant $C_2=C_2(d,s_1,s_2)$ such that: if $f \in \dot{H}^{s_1}(\mathbb{R}^d)$\; and\; $g\in\dot{H}^{s_2}(\mathbb{R}^d)$, then  $f.g \in \dot{H}^{s_1+s_2-1}(\mathbb{R}^d)$ and
		$$\|fg\|_{\dot{H}^{s_1+s_2-\frac{d}{2}}}\leq C_2 \|f\|_{\dot{H}^{s_1}}\|g\|_{\dot{H}^{s_2}}.$$
	\end{enumerate}
\end{lem}

\begin{lem}\label{lem45}
	Let $A,T>0$ and $f,g,h:[0,T]\rightarrow\R^+$ three continuous functions such that
	\begin{align}\label{e1}\;\;\;\;\forall t\in[0,T];\;f(t)+\int_0^tg(z)dz & \leq A+\int_0^th(z)f(z)dz.\end{align}
	Then $$\forall t\in[0,T];\;f(t)+\int_0^tg(z)dz\leq A\exp(\int_0^th(z)dz).$$
\end{lem}
\begin{proof} By Gronwall lemma, we get
	$$\forall t\in[0,T];\;f(t)\leq A\exp(\int_0^th(z)dz).$$
	Put this inequality in \ref{e1} we obtain
	$$\begin{array}{lcl}
		\displaystyle f(t)+\int_0^tg(z)dz&\leq&\displaystyle  A+\int_0^th(z)A\exp(\int_0^zh(r)dr)dz\\
		&\leq&\displaystyle  A+A\int_0^th(z)\exp(\int_0^zh(r)dr)dz\\
		&\leq&\displaystyle  A+A\int_0^t\Big(\exp(\int_0^zh(r)dr)\Big)'dz\\
		&\leq&\displaystyle  A+A\Big(\exp(\int_0^th(r)dr)-1\Big)\\
		&\leq&\displaystyle  A\exp(\int_0^th(r)dr),
	\end{array}$$
	which ends the proof.
\end{proof}
	\begin{lem}\label{lem46}
	Let $d\in\N$ Then, for all $x,y\in\R^d$, we have
	$$\langle\log(e+|x|^2)|x|^2x-\log(e+|y|^2)|y|^2y,x-y\rangle\geq0$$	
\end{lem}
\begin{proof}
	Let $a(z)=\log(e+|z|^2)|z|^2$ and $|x|\leq|y|$ :
	\begin{align*}
		\langle a(x)x-a(y)y,x-y\rangle=&\langle(a(x)-a(y))x,x-y\rangle+a(y)\langle x,x-y\rangle\\=&(a(x)-a(y))\langle x,x-y\rangle+a(y)|x-y|^{2}	
	\end{align*}	
	If $\langle x,x-y\rangle\geq0$ $(a(x)-a(y))\langle x,x-y\rangle+a(y)|x-y|^{2}\geq 0$\\Else,
	$\langle x,x-y\rangle<0$,	e.g:
	\begin{align*}
		(a(x)-a(y))\langle x,x-y\rangle+a(y)|x-y|^{2}	&=a(y)(|x-y|^{2}+\langle x,x-y\rangle) \\&=a(y)(\langle x-y,x-y\rangle-\langle x,x-y\rangle\\& a(y)\langle x-y,-y\rangle\\&=a(y)(|y|^2-\langle x,y\rangle)\\&\geq a(y)(|y|^2- |x||y|)\geq0 \end{align*}
	
\end{proof}
\begin{lem}\cite{JM}\label{lem5}
	Let $f:I\rightarrow\R$ be increasing function . Then there is $A\subset\R$ at most countable family such that for all $t$ in $A$,  $f$ is  discontinuous at $t$.
	Moreover, if $f$ is decreasing the $g=-f$.  	
\end{lem}

\section{\bf Existence and uniqueness of strong solution .}
\subsection{Proof of Theorem \ref{theo1}}
To begin, we integrate the $L^{2}$ scalar product of the first equation with $u$ on $[0,t] $, yielding
	 \begin{align}\label{eq01}\|u(t)\|_{L^2}^2+2\int_0^t\|\nabla_{h} u\|_{L^2}^2+2\alpha\int_0^t \|u\|_{L^{\beta+1}}^{\beta+1} \leq \|u^0\|_{L^2}^2\end{align}
	 Using the $\dot{H}^{0.1}$ scalar product in conjunction with $u$:
\begin{align*}
\frac{1}{2}\frac{d}{dt}	\|\partial_{3} u(t)\|_{L^2}^2+\|\nabla_{h}\partial_{3} u\|_{L^2}^2+\frac{\alpha(\beta-1)}{2} \||u|^{\beta-3} |\partial_{3}|u|^2|^2\|_{L^{1}}\\+\alpha\int_{\R^{3}}|u|^{\beta-1} \partial_{3} |u|^2\leq\langle\partial_{3}(u\nabla u),\partial_{3}u\rangle_{L^{2}}\\
\frac{1}{2}\frac{d}{dt}	\|\partial_{3} u(t)\|_{L^2}^2+\frac{1}{2}\|\nabla_{h}\partial_{3} u\|_{L^2}^2+\frac{\alpha(\beta-1)}{2} \||u|^{\beta-3} |\partial_{3}|u|^2|^2\|_{L^{1}}\\+\alpha\int_{\R^{3}}|u|^{\beta-1} \partial_{3} |u|^2\leq8\int_{\R^{3}}|u|^{2} \partial_{3} |u|^2.\end{align*}
Integrate on $[0,t]$ we get
\begin{align*}
	\|\partial_{3} u(t)\|_{L^2}^2+\int_0^t\|\nabla_{h}\partial_{3} u\|_{L^2}^2+\alpha(\beta-1) \int_{0}^{t}\||u|^{\beta-3} |\partial_{3}|u|^2|^2\|_{L^{1}}\\+2\alpha\int_0^t\||u|^{\beta-1} \partial_{3} |u|^2\|_{L^1}\leq\|\partial_{3}u^0\|_{L^2}+16\int_0^t\||u|^{\beta-1} \partial_{3} |u|^2\|_{L^1} .\end{align*}
$\bullet$For $\beta>3$ we obtain the global existence for bounded solution.\\
$\bullet$For $\beta=3$ Indeed, the problem is limited to the case $0<\alpha<8$ because the inequality (\ref{eqth02}) is unsolvable for these $\alpha$ values.
To solve our statement, we will add the function  $\log(e+|u|^2)$ to $|u|^{2}u$.
We will solve the incompressible Navier-Stokes equations with logarithmic damping \ref{theo2} at the next party
\subsection{Proof of Theorem $\ref{theo2}$}
$$$$
$\bullet$ {\bf{A priori estimates}}  \\ We start by taking the $L^{2}$ scalar product of the first equation with $u$, we get \begin{align}\label{eq3}\|u(t)\|_{L^2}^2+2\int_0^t\|\nabla_{h} u\|_{L^2}^2+2\alpha\int_0^t\|\log(e+|u|^2)|u|^4\|_{L^1}\leq \|u^0\|_{L^2}^2.\end{align}
Also, taking the $\dot{H}^{0,1}$ scalar product of $(NS_{\log})$ with $u$ :
\begin{align*}
	\frac{1}{2}\frac{d}{dt}\|\partial_{3} u\|^2+\|\nabla_{h}\partial_{3} u\|^2_{L^2}+\frac{\alpha}{2} \int_{\R^3} \frac{|u|^2}{e+|u|^2} |\partial_{3}|u|^2|^2&\\+\frac{\alpha}{2} \int_{\R^3} \log(e+|u|^2) |\partial_{3}|u|^2|^2+\alpha \int_{\R^3} \log(e+|u|^2)|u|^2 |\partial_{3} u|^2 &\leq |\langle \partial_{3} u\nabla u,\partial_{3}u \rangle|
\end{align*}
Since
\begin{align*}
	|\langle \partial_{3} u\nabla u,\partial_{3}u \rangle|&\leq\sum_{i=1}^{3}\int_{\R^3}|\partial_{3}u_{i}\partial_{i}u\partial_{3}u|
\end{align*}
Thus
\begin{align*}
	\frac{1}{2}\frac{d}{dt}\|\partial_{3} u\|_{L^2}^2+\|\nabla_{h}\partial_{3} u\|^2_{L^2}+\alpha \int_{\R^3} \log(e+|u|^2) |\partial_{3}|u|^2|^2&\leq 6\|\nabla_{h}\partial_{3}u\|_{L^2}\|u\partial_{3}u\|_{L^2}\\&\leq\frac{1}{2}\|\nabla_{h}\partial_{3}u\|^{2}_{L^2}+3\|u\partial_{3}u\|_{L^2}\\
	\frac{1}{2}\frac{d}{dt}\|\partial_{3} u\|_{L^2}^2+\frac{1}{2}\|\nabla_{h}\partial_{3} u\|^2_{L^2}+\alpha \int_{\R^3} \log(e+|u|^2) |\partial_{3}|u|^2|^2&\leq3\|u\partial_{3}u\|^{2}_{L^2}.
\end{align*}
For$t\leq0$, put the following set:
$$$$
Let $A_{t}=\{x\in\R^3/\alpha \log(e+|u|^2)-3\geq0\}.$
Since
\begin{align*}
	\int_{A^c_{t}}(\alpha \log(e+|u|^2)-3)|u_n|^2 |\nabla u|^2&\leq (e^{\frac{3}{2\alpha}}-e)_{+}\int_{A^c_{t}}|\partial_{3} u|^2.	
\end{align*}
So
\begin{align*}\frac{1}{2}\frac{d}{dt}\|\nabla_{h} u\|^{2}_{L^{2}}+\frac{1}{2}\|\Delta  u\|^{2}_{L^{2}}+\frac{\alpha}{2} \|\frac{|u|^2}{e+|u|^2} |\partial_{3}|u|^2|^2\|_{L^{1}}\\+\frac{\alpha}{2} \|\log(e+|u|^2) |\partial_{3}|u|^2|^2\|_{L^{1}}&\leq (e^{\frac{3}{2\alpha}}-e)_{+}\int_{A^c_{t}}|\partial_{3} u|^2\\&\leq(e^{\frac{3}{2\alpha}}-e)_{+}\int_{\R^3}|\partial_{3} u|^2\\&\leq b_{\alpha}\int_{\R^3}|\partial_{3} u|^2, \end{align*}
where $b_{\alpha}=(e^{\frac{3}{2\alpha}}-e)_{+}$ \\
Integrate on $[0,T]$,we get:
\begin{align*}
	&	\|\partial_{3} u(t)\|_{L^2}^2+\int_{0}^{t}\|\nabla_{h}\partial_{3} u\|^2_{L^2}+\alpha \int_{0}^{t} \|\frac{|u|^2}{e+|u|^2} |\partial_{3}|u|^2|^2\|_{L^{1}}+\alpha \int_{0}^{t} \|\log(e+|u|^2) |\partial_{3}|u|^2|^2\|_{L^{1}}\\
	&\hskip1cm+\alpha \int_{0}^{t} \|\log(e+|u|^2)|u|^2 |\partial_{3} u|^2\|_{L^{1}}\leq\|\partial_{3} u^0\|_{L^2}^2+ b_{\alpha}\int_{0}^{t}	 \|\partial_{3}u\|^{2}_{L^2}.
\end{align*}
By Gronwall Lemma and (\ref{lem5}) we obtain :
\begin{align}\label{eq4}
	\nonumber	\|\partial_{3} u(t)\|_{L^2}^2+\int_{0}^{t}\|\nabla_{h}\partial_{3} u\|^2_{L^2}+\alpha \int_{0}^{t} \|\frac{|u|^2}{e+|u|^2} |\partial_{3}|u|^2|^2\|_{L^{1}}&\\+\alpha \int_{0}^{t} \|\log(e+|u|^2) |\partial_{3}|u|^2|^2\|_{L^{1}}+\alpha \int_{0}^{t} \|\log(e+|u|^2)|u|^2 |\partial_{3} u|^2\|_{L^{1}}&\leq\|\partial_{3} u(t)\|_{L^2}^2 e^{b_{\alpha t}}.
\end{align}
Absolutely, these bounds come from the approximate solutions via the Friederich's regularization procedure. 
The passage to the limit follows using classical argument by combining Ascoli's Theorem and the Cantor Diagonal Process \cite{HB}. And this solution is in $L^{\infty}(\R^{+},H^{0,1}(\R^3))$ such that $\nabla_{h}u\in L^{2}_{loc}(\R^{+},H^{0,1}(\R^3))$ (\ref{eq3}) and (\ref{eq4}).
$\bullet${\bf{Uniqueness :}}
$$$$
This proof is inspired by \cite{HK}. Let $u,v$ two solutions of $(NSD_{\log})$ and $w=u-v$
$$	\partial_t u
-\Delta u+ u.\nabla u  +\alpha \log(e+|u|^2)|u|^2u =\;\;-\nabla p_{1}\hbox{ in } \mathbb R^+\times \mathbb R^3~~~~~~~~~~(1)$$
$$\partial_t v
-\Delta v+ v.\nabla v +\alpha \log(e+|v|^2)|v|^2v =\;\;-\nabla p_{2}\hbox{ in } \mathbb R^+\times \mathbb R^3~~~~~~~~~~~(2).$$

We make the difference $(1)-(2)$,we get: $$\partial_t w -\Delta w+w.\nabla u+v\nabla w +\alpha(\log(e+|u|^2)|u|^2u-\log(e+|v|^2)|v|^2u)=-\nabla (p_1-p_2).$$
Taking the $L^2$ scalar product, we have :
\begin{align*}
	\frac{1}{2}\frac{d}{dt}\|w\|^{2}_{L^2}+\|\nabla w\|^{2}_{L^2} +\alpha\langle(\log(e+|u|^2)|u|^2u-\log(e+|v|^2)|v|^2v),w\rangle&\leq|\langle w\nabla u, w\rangle|_{L^2}.
\end{align*}
Using Lemma \ref{lem46}, we get:
\begin{align*}
	\frac{1}{2} \frac{d}{dt} \|w\|^{2}_{L^{2}}+\|\nabla_{h} w\|^{2}_{L^{2}}&\leq| \int_{\R^3} (w.\nabla u).w dx|.
\end{align*}	
But $$\int_{\R^3} (w.\nabla u).w dx=\sum_{i=1}^{2}\sum_{j=1}^{3}\int_{\R^3}w_{i}\partial_{i}u_{j}w_{j}+\sum_{j=1}^{3}\int_{\R^3}w_{3}\partial_{3}u_{j}w_{j}=F_{1}+F_{2}.$$
By H\"{o}lder inequality, we get :
$$F_{1}\leq \sum_{i=1}^{2}\sum_{j=1}^{3} \|\partial_{i}u_{j}\|_{L_{v}^\infty L_{h}^2}\|w_{i}\|_{L_{v}^2 L_{h}^4}\|w_{j}\|_{L_{v}^2 L_{h}^4}.$$
Since $\dot{H}^\frac{1}{2}(\R^2)\hookrightarrow L^4(\R^2)$ and by interpolation we get:
$$\|w_{j}\|_{L_{h}^4}\leq c\|w_{j}\|_{L^2}^{\frac{1}{2}}\|\nabla w_{j}\|_{L^2}^{\frac{1}{2}},$$ so $$\|w_{j}\|_{L^2 L_{h}^4}\leq c \|w_{j}\|_{L^2}^{\frac{1}{2}}\|\nabla_h w_{j}\|_{L^2}^{\frac{1}{2}}.$$
We have:
$$\|\partial_{i}u_{j}\|^2=\int_{-\infty}^{x_3}\frac{d}{dz}\|\partial_{i}u_{j}\|^2dz=2\int_{-\infty}^{x_3}\partial_{z} \partial_{i}u_{j}\partial_{i}u_{j} dz\leq \|\partial_{3} \partial_{i}u_{j}\|_{L^2}\|\partial_{i}u_{j}\|_{L^2}. $$
Then $$F_{1}\leq c\|\partial_{3} \nabla_{h}u\|^{\frac{1}{2}}_{L^2}\|\nabla_{h}u\|^{\frac{1}{2}}_{L^2}\|w\|_{L^2}|\nabla_{h} w\|_{L^2}. $$
By Young inequality, we obtain:
\begin{align}\label{f1}F_{1}\leq \frac{1}{4}\|\nabla_{h} w\|^2_{L^2}+\frac{c}{4}(\|\partial_{3} \nabla_{h}u\|^{2}_{L^2}+\|\nabla_{h}u\|^{2}_{L^2})\|w\|^2_{L^2}. \end{align}
The same procedure for $F_{2}$ we get :
\begin{align*}
	F_{2}\leq \|w_3\|_{L_{v}^\infty L_h^2}\|\nabla_{h} w\|^{\frac{1}{2}}_{L^2}\|\partial_{3} \nabla_{h}u\|^{\frac{1}{2}}_{L^2}\|\nabla_{h}u\|_{L^2}\|w\|_{L^2} .
\end{align*}
Since
$$ \|w_3\|_{L_{v}^\infty L_h^2}=2\int_{-\infty}^{x_3}\int_{\R^2}w_3(x_h,z)\partial_{3}w_3(x_h,z)dx_hdz.$$ Using the fact that $\nabla.w=0$ so ${\rm div_h} w_h=-\partial_{3}w_3$ and
$$ \|w_3\|_{L_{v}^\infty L_h^2}=-2\int_{-\infty}^{x_3}\int_{\R^2} w_3(x_h,z){\rm div_h} w_h w_3(x_h,z)dx_h dz \leq 2\|{\rm div_h} w_h\|_{L^2}\|w_3\|_{L^2}.$$
By Young inequality, we obtain:
\begin{align}\label{f2}
	F_{2}\leq \frac{1}{4}\|\nabla_{h} w\|^2_{L^2}+\frac{c}{4}(\|\partial_{3} \nabla_{h}u\|^{2}_{L^2}+\|\partial_{3}u\|^{2}_{L^2})\|w\|^2_{L^2}.
\end{align}
Hence, according to ($\ref{f1}$) and (\ref{f2}) we get:
\begin{align*}
	\frac{1}{2} \frac{d}{dt} \|w\|^{2}_{L^{2}}+\frac{1}{2}\|\nabla_{h} w\|^{2}_{L^{2}}+\alpha\langle(\log(e+|u|^2)|u|^2u-\log(e+|v|^2)|v|^2v),w\rangle&\leq c\|\partial_{3} \nabla_{h}u\|^{2}_{L^2}\|w\|^2_{L^2} \\+c(\|\partial_{3}u\|^{2}_{L^2}+\|\nabla_{h}u\|^{2}_{L^2})\|w\|^2_{L^2}.
\end{align*}
Integrate on $[0,t]$, we have :
\begin{align*}
	\|w(t)\|^{2}_{L^{2}}+\int_{0}^{t}\|\nabla_{h} w\|^{2}_{L^{2}}&\leq\|w(0)\|^{2}_{L^{2}} +c\int_{0}^{t}(\|\partial_{3} \nabla_{h}u\|^{2}_{L^2}+\|\partial_{3}u\|^{2}_{L^2}+\|\nabla_{h}u\|^{2}_{L^2})\|w\|^2_{L^2} .
\end{align*}
Then, by Gronwall Lemma :
\begin{align}\label{M}
	\|w(t)\|^{2}_{L^{2}}\leq\|w(0)\|^{2}_{L^{2}} e^{c't}. \end{align}
But $w(0)=u(0)-v(0)=0$, then $u=v$.\\
	$\ast${\bf{Right continuity:}}
$$$$
$\ast$ Right continuity at $0$:
Let $t_{k}>0$ such that $t_{k}\underset{k\rightarrow\infty}\rightarrow 0 $ then
\begin{align}\underset{k\rightarrow\infty}{\limsup}\|\partial_{3} u(t_{k})\|^{2}_{L^{2}}\leq \|\partial_{3} u^{0}\|_{L^2}^{2}.\end{align}
For $(3.7)$ we have the Right continuity at $0$.\\
$\ast$ Right continuity at $t_{0}$:
Let $t_{0}>0$
$$
\begin{cases}
	\partial_t v
	-\Delta_{h} v+ v.\nabla v  +\alpha \log(e+|u|^2)|u|^2u =\;\;-\nabla p\hbox{ in } \mathbb R^+\times \mathbb R^3\\
	{\rm div}\, v = 0 \hbox{ in } \mathbb R^+\times \mathbb R^3\\
	v(0,x) =u(t_{0},x) \;\;\hbox{ in }\mathbb R^3.
\end{cases}
$$
By {\bf uniqueness} of solution $v(t)=u(t+t_{0})$ moreover $u$ is continuous on the right at $0$. Then $u$ is continuous on the right at $t_{0}$.\\
$\ast$ Continuity of $(NSAn_{\log})_{\alpha}$ in $H^{0,1}$:\\Let $t_{1}\leq t_{2}$\\
$\|\nabla u\|_{L^{2}}$ is continuous on $\R^{+}\backslash A$, where
$A=\{t\in\R+/f ~~~\mbox{discontinous at} ~~~t\}$ is at most countable set with $f(t)=e^{-tb_{\alpha}}\|\partial_{3} u(t)\|^{2}_{L^2}$, since
$$ \|\partial_{3} u(t_{2})\|^{2}_{L^2}\leq\|\partial_{3} u(t_{1})\|^{2}_{L^2} e^{b_{\alpha}(t_{2}-t_{1})}$$
we get
$$\|\partial_{3} u(t_{2})\|^{2}_{L^2}e^{-b_{\alpha}t_{2}}\leq \|\partial_{3} u(t_{1})\|^{2}_{L^2} e^{-b_{\alpha}t_{1}}.$$
Thus, $f$ is a decreasing function. According to $(\ref{lem5})$, $f$ is continuous on $\R^{+}\backslash A$.
\section{\bf Appendix.}
In this part, we give a simple proof of $u\in C(\R^{+},L^{2}(\R^3))$, which is inspired by\cite{JM},  where $u$ is a solution of $(NSAn_{\log})_{\alpha}$ given by Friederich approximation.\\
By inequality  ($\ref{eqth2}$) we get
$$\limsup_{t\rightarrow0}\|u(t)\|_{L^{2}}\leq\|u^{0}\|_{L^{2}}.$$
Thus, Proposition \ref{prop1}-(3) implies that
$$\limsup_{t\rightarrow0}\|u(t)-u^0\|_{L^{2}}=0.$$
which ensures the continuity at 0.\\
$\bullet$ Let $t_{0}>0$. For $\delta\in(0,t_{0})$ and $n\in\N$, put the following function $$u_{n,\delta}=u_{\varphi(n)}(t+\delta).$$
Applying the same method to prove the uniqueness to $u_{\varphi(n)}$ and using (\ref{M}) we get
$$\|u_{\varphi(n)}(t+\delta)-u_{\varphi(n)}(t)\|^{2}_{L^{2}}\leq \|u_{\varphi(n)}(\delta)-u_{\varphi(n)}(0)\|^{2}_{L^{2}}\exp(cF_{n}(t)),$$
where $$F_{n}(t)=\int_{0}^{t}(\|\partial_{3} \nabla_{h}u_{\varphi(n)}\|^{2}_{L^2}+\|\partial_{3}u_{\varphi(n)}\|^{2}_{L^2}+\|\nabla_{h}u_{\varphi(n)}\|^{2}_{L^2}).$$
By using inequalities (\ref{eq3}) and (\ref{eq4}), we get
\begin{align*}
	F_{n}(t)&\leq\|\partial_{3}u^{0}\|^{2}_{L^2}e^{b_{\alpha}t}	 +\|\partial_{3}u^{0}\|^{2}_{L^2}\frac{e^{b_{\alpha}t}-1}{b_{\alpha}}+\frac{\|u^{0}\|^{2}_{L^2}}{2}\\
&\leq(1+\frac{1}{b_{\alpha}})\|\partial_{3}u^{0}\|^{2}_{L^2}e^{2b_{\alpha}t_{0}}+\frac{\|u^{0}\|^{2}_{L^2}}{2}.
\end{align*}	
For $t\in[0,2t_{0}]$, we have:
$$F_{n}(t)\leq(1+\frac{1}{b_{\alpha}})\|\partial_{3}u^{0}\|^{2}_{L^2}e^{2b_{\alpha}t_{0}}+\frac{\|u^{0}\|^{2}_{L^2}}{2}=M_{\alpha}(t_{0},u^0).$$
Then for $t=t^0$ and $t=t_{0}-\delta$, we get:
\begin{align}\label{A}
	\|u_{\varphi(n)}(t+\delta)-u_{\varphi(n)}(t)\|^{2}_{L^{2}}&\leq \|u_{\varphi(n)}(\delta)-u_{\varphi(n)}(0)\|^{2}_{L^{2}}\exp(cM_{\alpha}(t_{0},u^0)).
\end{align}	
\begin{align}\label{B}
	\|u_{\varphi(n)}(t-\delta)-u_{\varphi(n)}(t)\|^{2}_{L^{2}}&\leq \|u_{\varphi(n)}(\delta)-u_{\varphi(n)}(0)\|^{2}_{L^{2}}\exp(cM_{\alpha}(t_{0},u^0)).
\end{align}	
The idea is to lower the terms on the left and increase the term on the right of the inequalities (\ref{A}) and (\ref{B}).\\
For the right term, we write
\begin{align*}
	 \|u_{\varphi(n)}(\delta)-u_{\varphi(n)}(0)\|^{2}_{L^{2}}&=\|u(\delta)\|^{2}_{L^{2}}+\|u(0)\|^{2}_{L^{2}}-2Re \langle u_{\varphi(n)}(\xi)-u_{\varphi(n)}(0) \rangle_{L^2}.
\end{align*}
By using inequality (\ref{eqth1}), we obtain:
\begin{align*}
	\|u_{\varphi(n)}(\delta)-u_{\varphi(n)}(0)\|^{2}_{L^{2}}&\leq2\|u^0\|^{2}_{L^{2}}-2Re \langle u_{\varphi(n)}(\delta)-u_{\varphi(n)}(0) \rangle_{L^2}\\&\leq2\|u^0\|^{2}_{L^{2}}-2Re \langle u_{\varphi(n)}(\delta),u^0\rangle_{L^{2}}-2Re \langle u_{\varphi(n)}(\delta), u_{\varphi(n)}(0)-u^0\rangle_{L^{2}}.
\end{align*}
But
$$|\langle u_{\varphi(n)}(\delta), u_{\varphi(n)}(0)-u^0\rangle_{L^{2}}|\leq \|u_{\varphi(n)}(\delta)\|_{L^2}\|u_{\varphi(n)}(0)-u^0\|_{L^2},$$	
then $$\lim_{n\rightarrow\infty}|\langle u_{\varphi(n)}(\delta), u_{\varphi(n)}(0)-u^0\rangle_{L^{2}}|=0.$$
On the other hand, and by using that $u_{\varphi(n)}(\delta)$ converge weakly in $L^2(\R^3)$ to $u(\delta),$
we get $$\liminf_{n\rightarrow\infty}\|u_{\varphi(n)}(\delta)-u_{\varphi(n)}(0)\|_{L^2}\leq2\|u^{0}\|^{2}_{L^{2}}-2Re\langle u_{\varphi(n)}(\delta),u^0\rangle_{L^2}.$$
For the left term, we have, for all $q,N\in\N$
\begin{align*} \|J_{N}(\theta_{q}.(u_{\varphi(n)}(t\pm\delta)-u_{\varphi(n)}(t)))\|^{2}_{L^{2}}&\leq
\|\theta_{q}.(\theta_{q}.(u_{\varphi(n)}(t\pm\delta)-u_{\varphi(n)}(t)))\|^{2}_{L^{2}}\\
&\leq \|u_{\varphi(n)}(t\pm\delta)-u_{\varphi(n)}(t)\|^{2}_{L^{2}}.
\end{align*}
Using the fact that $$\lim_{n\rightharpoonup\infty}\|\theta_{q}(u_{\varphi(n)}-u)\|_{L^\infty([0,T_{q}],H^{-4})}=0,$$
we get:
\begin{align*}
	\|J_{N}(\theta_{q}.(u(t\pm\delta)-u(t)))\|^{2}_{L^{2}}\leq 2(\|u\|^{2}_{L^{2}}-Re\langle u(\delta),u^0\rangle_{L^2})\exp(cM_{\alpha}(t_{0},u^0)).
\end{align*}
By applying the Monotonic Convergence Theorem in the order $N\rightarrow0$ and $q\rightarrow\infty$ we get:
\begin{align*}
	\|\theta_{q}.(u(t\pm\delta)-u(t))\|^{2}_{L^{2}}\leq 2(\|u\|^{2}_{L^{2}}-Re\langle u(\delta),u^0\rangle_{L^2})\exp(cM_{\alpha}(t_{0},u^0)).
\end{align*}
Using the continuity at $0$ and make $\delta\rightarrow0$, we get the continuity at $t_{0}$, which ends the proof.


\begin{thebibliography}{10}
	\bibitem{HK} H. Bessaih, S. Trabelsi and H.Zorgati {\it Existence and uniqueness of global solutions for the modified anisotropic 3D Navier-Stockes equations\/} ,
	\bibitem{HB} H. Bahouri, J.Y Chemin and R. Danchin, {\it Fourier Analysis and Nonlinear Partial Differential
		Equations\/}, Springer Verlag, 523p, 2011.
	\bibitem{HBAF} H. Brezis, {\it Analyse Fonctionnel: Th\'eorie et applications\/}, Masson, 234p, (1996).
	\bibitem{JM} J. Benameur and M. Ltifi, Strong solution of 3D-NSE with exponential damping, arXiv:2103.16707,2021.
	\bibitem{JYC} {J.-Y. Chemin},
	\emph{About Navier-Stokes equations,}
	{Publications of Jaques-Louis Lions Laboratoiry, Paris VI University, R96023, (1996)}.
	\bibitem{XJ} X. Cai and Q. Jiu, {\it Weak and strong solutions for the incompressible Navier-Stokes with damping\/}, Journal of Mathematical Analysis and Applications, 343, p 799-809, 2008.
	\bibitem{JP} J.Pedlosky, {\it Geophysical Fluids Dynamics\/}. Springer Verlag, New York (1987).
	\bibitem{R} R. Bennacer, A. Tobbal and H. Beji, Convection naturelle Thermosolutale dans une Cavit´e Poreuse Anisotrope: Formulation	de Darcy-Brinkman. Rev. Energ. Ren. 5 (2002) 1–21.
	
\end{thebibliography}
\end{document}